\documentclass[a4paper,11pt]{article}
\usepackage{amsmath,amsthm,amssymb,amsfonts,fullpage}

\newtheorem{thm}{Theorem}[section]

\newtheorem{cor}[thm]{Corollary}

\newtheorem{lem}[thm]{Lemma}
\newtheorem{prop}[thm]{Proposition}
\newtheorem{rem}[thm]{Remark}

\begin{document}

\title{Scaling limit for random walk\\on the range of random walk\\in four dimensions}
\author{D.~A.~Croydon and D.~Shiraishi}
\maketitle

\begin{abstract}
We establish scaling limits for the random walk whose state space is the range of a simple random walk on the four-dimensional integer lattice. These concern the asymptotic behaviour of the graph distance from the origin and the spatial location of the random walk in question. The limiting processes are the analogues of those for higher-dimensional versions of the model, but additional logarithmic terms in the scaling factors are needed to see these. The proof applies recently developed machinery relating the scaling of resistance metric spaces and stochastic processes, with key inputs being natural scaling statements for the random walk's invariant measure, the associated effective resistance metric, the graph distance, and the cut times for the underlying simple random walk.
%60K37   	Processes in random environments
%05C81   	Random walks on graphs
%60K35      Interacting random processes; statistical mechanics type models; percolation theory
%82B41      Random walks, random surfaces, lattice animals, etc. in equilibrium statistical mechanics
\end{abstract}

\section{Introduction}

The object of interest in this article is the discrete-time simple random walk on the range of a four-dimensional random walk. To introduce this, we follow the presentation of \cite{Croy}, in which the same and higher-dimensional versions of the model were studied. Let $S=(S_n)_{n\geq 0}$ be the simple random walk on $\mathbb{Z}^4$ started from $0$, built on an underlying probability space $\Omega$ with probability measure $\mathbf{P}$. Define the range of the random walk $S$ to be the graph $\mathcal{G}:=(V(\mathcal{G}),E(\mathcal{G}))$ with vertex set
\begin{equation}\label{vdef}
V(\mathcal{G}):=\left\{S_n:\:n\geq 0\right\},
\end{equation}
and edge set
\begin{equation}\label{edef}
E(\mathcal{G}):=\left\{\{S_n,S_{n+1}\}:\:n\geq 0\right\}.
\end{equation}
For $\mathbf{P}$-a.e.\ random walk path, the graph $\mathcal{G}$ is infinite, connected and has bounded degree. Given a realisation of $\mathcal{G}$, we write
\[X=\left((X_n)_{n\geq 0},(\mathbf{P}_x^\mathcal{G})_{x\in V(\mathcal{G})}\right)\]
to denote the discrete-time Markov chain with transition probabilities
\[P_\mathcal{G}(x,y):=\frac{1}{\mathrm{deg}_\mathcal{G}(x)}\mathbf{1}_{\{\{x,y\}\in E(\mathcal{G})\}},\]
where $\mathrm{deg}_\mathcal{G}(x)$ is the usual graph degree of $x$ in $\mathcal{G}$. For $x\in V(\mathcal{G})$, the law $\mathbf{P}_x^\mathcal{G}$ is the quenched law of the simple random walk on $\mathcal{G}$ started from $x$. Since $0$ is always an element of $V(\mathcal{G})$, we can also define an averaged (or annealed) law $\mathbb{P}$ for the random walk on $\mathcal{G}$ started from 0 as the semi-direct product of the environment law $\mathbf{P}$ and the quenched law $\mathbf{P}_0^\mathcal{G}$ by setting
\[\mathbb{P}:=\int \mathbf{P}_0^\mathcal{G}(\cdot)\mathrm{d}\mathbf{P}.\]

Our main results are the following averaged scaling limits for the simple random walk $X$ on $\mathcal{G}$. The processes $(B_t)_{t\geq 0}$ and $(W_t)_{t\geq 0}$ are assumed to be independent standard Brownian motions on $\mathbb{R}$ and $\mathbb{R}^4$ respectively, both started from the origin of the relevant space. The notation $d_\mathcal{G}$ is used to represent the shortest path graph distance on $\mathcal{G}$. We note that it was previously shown in \cite[Theorem 1.3]{Croy} that $n^{-1/4}|X_n|$ is with high probability of an order in the interval $[(\log n)^{1/24},(\log n)^{7/12}]$, and part (b) refines this result substantially; a reparameterisation yields that the correct order of $n^{-1/4}|X_n|$ is $(\log n)^{1/8+o(1)}$. Similarly, a reparameterisation of part (a) yields that the correct order of $n^{-1/2}d_\mathcal{G}(0,X_n)$ is $(\log n)^{-1/4+o(1)}$.

\begin{thm}\label{main1} (a) There exist deterministic slowly-varying functions $\phi:\mathbb{N}\rightarrow \mathbb{R}$ and $\psi:\mathbb{N}\rightarrow \mathbb{R}$ satisfying
\[\phi(n)=\left(\log n\right)^{-\frac{1}{2}+o(1)},\qquad\psi(n)=\left(\log n\right)^{-\frac{1}{2}+o(1)},\]
such that the law of
\[\left(\frac{1}{n\phi(n)}d_\mathcal{G}\left(0,X_{\lfloor tn^2\psi(n)\rfloor}\right)\right)_{t\geq 0}\]
under $\mathbb{P}$ converges as $n\rightarrow\infty$ to the law of $(|B_t|)_{t\geq 0}$.\\
(b) For $\psi:\mathbb{N}\rightarrow \mathbb{R}$ as in part (a), the law of
\[\left(\frac{1}{n^{1/2}}X_{\lfloor tn^2\psi(n)\rfloor}\right)_{t\geq 0}\]
under $\mathbb{P}$ converges as $n\rightarrow\infty$ to the law of $(W_{|B_t|})_{t\geq 0}$.
\end{thm}

It was shown in \cite{Croy} that the corresponding result is true in dimensions $d\geq 5$, but without the need for the $\phi$ and $\psi$ terms in the scaling factors. (Actually, \cite{Croy} established the convergence of part (a) in a stronger way, namely under the quenched law $\mathbf{P}_0^\mathcal{G}$, for $\mathbf{P}$-a.e.\ realisation of $\mathcal{G}$.) As will become clear in our argument, the appearance of $\phi$ and $\psi$ in the four-dimensional case stem from their appearance in the scaling of the graph distance and effective resistance between $0$ and $S_n$ (see Theorem \ref{main2} below). In the lower-dimensional cases, the situation is substantially different. Indeed, for $d=1,2$, the underlying random walk is recurrent, and so $\mathcal{G}$ is simply $\mathbb{Z}^d$ equipped with all nearest-neighbour edges. Thus, the random walk $X$ simply scales diffusively to Brownian motion on $\mathbb{R}^d$. As for the case $d=3$, we do not have a conjecture for the scaling limit of the process. However, the random walk $S$ and its scaling limit $W$ are both transient, and therefore neither the graph $\mathcal{G}$, nor its scaling limit, will be space filling. On the other hand, we know that the graph $\mathcal{G}$ no longer has asymptotically linear volume growth (see \cite[Proposition 4.3.1]{Shk}), has spectral dimension strictly greater than one (see \cite[Theorem 1.2.3]{Shk}), and has loops that persist in the limit (see \cite{DEK} for the original proof that Brownian motion has double points, and \cite{LT,SS} for a description of how the random walk `loop soup' converges to that of Brownian motion). Therefore, in three dimensions, we do not expect scaling limits for $X$ as above.

As was noted in \cite{Croy}, one motivation for studying random walk on the range of a random walk is its application to the transport properties of sedimentary rocks of low porosity, where the commonly considered sub-critical percolation model does not reflect the pore connectivity properties seen in experiments, see \cite{Ban}. (We note that, in \cite{Ban}, it was observed that the effective resistance in $\mathcal{G}$ between 0 and $S_n$ is of order $n(\log n)^{-1/2}$, as was later rigourously proved in \cite{S}.) Further motivation given in \cite{Croy} was that the model provides a simple prototypical example of a graph with loops that nonetheless has a tree-like scaling limit, for which one could establish convergence of the associated random walks to the natural diffusion on the limiting tree-like space. Such a situation is also expected to be the case for more challenging models, such as critical percolation in high dimensions (see \cite{BCF1,BCF2,Chaus} for progress in this direction). In addition to works specifically related to high-dimensional percolation, recent years have seen the development of a general approach for deriving scaling limits of random walks on graphs with a suitably `low-dimensional' asymptotic structure, which demonstrates that if the effective resistance and invariant measure of the random walk converge in an appropriate fashion, then so does the random walk itself (see \cite{Croyres,CHK}, and \cite{ALW} for the particular case of tree-like spaces). We will appeal to this framework to prove Theorem \ref{main1}, with relevant details being introduced in Section \ref{sec3}. We highlight that the present work is a first application of such `resistance form' theory to a model at its critical dimension, where logarithmic terms appear in the scaling factors. (NB.\ Although the viewpoint of \cite{Croy} was a little different, combining the basic estimates of \cite{Croy} with the argument of \cite{Croyres} would also yield the scaling limits of random walk on the range of random walk in dimensions greater than or equal to 5 as given in \cite{Croy}.)

Providing the key inputs into the resistance form argument that we apply to deduce Theorem \ref{main1}, the first part of the paper is devoted to establishing a number of basic properties about the underlying graph $\mathcal{G}$. These give appropriate notions of volume, resistance, graph distance, and cut-time scaling for $S$. We highlight that much of the work needed to establish the following result was completed in earlier articles, see Remark \ref{back}(a) for details. To state the result, we need to define a number of quantities. Firstly, let $\mu_\mathcal{G}$ be the measure on $V(\mathcal{G})$ given by
\[\mu_{\mathcal{G}}(\{x\}):=\mathrm{deg}_\mathcal{G}(x),\qquad \forall x\in V(\mathcal{G});\]
this is the unique (up to scaling) invariant measure of the Markov chain $X$. Secondly, let $R_{\mathcal{G}}$ be the effective resistance metric on $V(\mathcal{G})$ when we consider $\mathcal{G}$ to be an electrical network with unit resistors placed along each edge (see \cite{DS,LPW} for elementary background on the effective resistance, and \cite[Chapter 2]{Barlow} for a careful treatment of effective resistance for infinite graphs). We have already introduced the shortest path graph distance $d_\mathcal{G}$. Finally, we write
\begin{equation}\label{cuttimes}
\mathcal{T}:=\left\{n:\:S_{[0,n]}\cap S_{[n+1,\infty)}=\emptyset\right\}
\end{equation}
for the set of cut times of $S$, where $S_{[a, b]}:= \{ S_{k}:\: a \le k \le b \}$ and $S_{[a, \infty)} := \{ S_{k}:\:k \ge a \}$ for $0 \le a \le b < \infty$. It is known that $\mathcal{T}$ is $\mathbf{P}$-a.s.\ an infinite set (see \cite{cut}), and we denote the elements of this set as $0\leq T_1<T_2<\dots$.

\begin{thm}\label{main2} For any $T\in(0,\infty)$, the following limits hold in $\mathbf{P}$-probability as $n\rightarrow\infty$.\\
(a) For some deterministic constant $\lambda\in(0,\infty)$,
\[\left(\frac{\mu_\mathcal{G}\left(S_{[0,nt]}\right)}{n}\right)_{t\in[0,T]}\rightarrow \left(\lambda t\right)_{t\in[0,T]}.\]
(b) For some deterministic slowly-varying function $\tilde{\psi}:\mathbb{N}\rightarrow \mathbb{R}$ satisfying $\tilde{\psi}(n)=(\log n)^{-\frac{1}{2}+o(1)}$,
\[\left(\frac{R_\mathcal{G}(0,S_{nt})}{n\tilde{\psi}(n)}\right)_{t\in[0,T]}\rightarrow \left(t\right)_{t\in[0,T]}.\]
(c) For some deterministic slowly-varying function $\phi:\mathbb{N}\rightarrow \mathbb{R}$ satisfying $\phi(n)=(\log n)^{-\frac{1}{2}+o(1)}$,
\[\left(\frac{d_\mathcal{G}(0,S_{nt})}{n\phi(n)}\right)_{t\in[0,T]}\rightarrow \left(t\right)_{t\in[0,T]}.\]
(d) For some deterministic constant $\tau\in(0,\infty)$,
\[\left(\frac{T_{nt}}{n(\log n)^{\frac{1}{2}}}\right)_{t\in[0,T]}\rightarrow \left(\tau t\right)_{t\in[0,T]}.\]
\end{thm}

\begin{rem}\label{back}
(a) As is claimed in \cite{cut}, the techniques of \cite[Chapter 7]{Lawb} yield part (d) of the above result. In order for completeness, we provide details below. As for part (a), this is new, though its proof is also based on ideas from \cite{Lawb}. Concerning part (b), the one-dimensional marginal of the process in question was dealt with in \cite{S}. Here, we do the additional work to derive the functional convergence statement. (Since the resistance distance from 0 to $S_n$ is not monotone increasing, this is not a completely trivial exercise.) Part (c) is handled in the same way as part (b).\\
(b) The $\psi$ of Theorem \ref{main1} is given by $\lambda\times \tilde{\psi}$, where $\lambda$ and $\tilde{\psi}$ are given above.\\
(c) We conjecture that it is possible to take $\phi(n)=c_1(\log n)^{-1/2}$, $\tilde{\psi}(n)=c_2(\log n)^{-1/2}$ in the above result (and therefore also $\psi(n)= c_3(\log n)^{-1/2}$ in Theorem \ref{main1}).
\end{rem}

Apart from the scaling limits of Theorem \ref{main1}, given Theorem \ref{main2}, various consequences for the simple random walk $X$ stem from general techniques of \cite{LLT} and \cite{KM}. In particular, we have the following results describing the growth rate of the quenched exit times and the scaling limit of the heat kernel of the process. To state these, we introduce the notation
\[\tau_r^\mathcal{G}:=\inf\left\{n\geq 0:\:d_\mathcal{G}(0,X_n)\geq r\right\}\]
for the time that $X$ exits a $d_\mathcal{G}$-ball of radius $r$ centred at the origin, and
\[p_{n}^\mathcal{G}\left(x,y\right):=\frac{\mathbf{P}_x^\mathcal{G}\left(X_n=y\right)+\mathbf{P}_x^\mathcal{G}\left(X_{n+1}=y\right)}{2\mathrm{deg}_\mathcal{G}(y)}\]
for the (smoothed) transition density of $X$. We note that \cite[Theorem 1.4]{Croy} showed that $p_{n}^\mathcal{G}(0,0)$ deviated from $n^{-1/2}$ by a logarithmic amount, but did not pin down the precise exponent of the logarithm.

\begin{cor}\label{srwcor} (a) It holds that
\[\lim_{\Lambda\rightarrow\infty}\inf_{r\geq 1}\mathbf{P}\left(\Lambda^{-1}\leq\frac{\mathbf{E}_0^\mathcal{G}\left(\tau^\mathcal{G}_r\right)}{r^2\tilde{\psi}(r)\phi(r)^{-2}}\leq\Lambda\right)=1.\]
(b) For any compact interval $I\subseteq (0,\infty)$ and $x_0\in[0,\infty)$, it holds that
\[\sup_{t\in I}\sup_{x\in[0,x_0]}\left|\lambda np_{\lfloor tn^2\psi(n)\rfloor}^\mathcal{G}\left(0,S_{\lfloor nx\rfloor}\right)-\sqrt{\frac{2}{\pi t}}e^{-\frac{x^2}{2t}}\right|\rightarrow 0,\]
in $\mathbf{P}$-probability.
\end{cor}

\begin{rem}
In \cite[Theorem 1.2.2]{S}, it was stated that there exist constants $c_1,c_2\in(0,\infty)$ such that, for $\mathbf{P}$-a.e.\ realisation of $\mathcal{G}$,
\[c_1n^{-1/2}\tilde{\psi}(n)^{1/2}\leq p_{n}^\mathcal{G}\left(0,0\right)\leq c_2n^{-1/2}\tilde{\psi}(n)^{1/2}\]
for all large $n$. (Note that, in \cite{S}, $\tilde{\psi}(n)$ was defined to be equal to $\mathbf{E}(R_\mathcal{G}(0,S_n))/n$, and this is consistent with the current article, since the conclusion of Theorem \ref{main2}(b) holds for this particular choice of $\tilde{\psi}$.) However, the proof in \cite{S} incorporated a misapplication of \cite[Proposition 3.2]{KM}, in that it did not include a verification of the resistance lower bound of \cite[Equation (3.3)]{KM}. This gap is filled in \cite{S-erratum}, which applies the estimates of \cite{S} to obtain the missing resistance estimate. The latter can be considered a $\mathbf{P}$-a.s.\ version of Proposition \ref{p2} below (or a quantitative version of \eqref{rcond}).
\end{rem}

The remainder of the paper is organised as follows. In Section \ref{sec2}, we derive the fundamental estimates on the underlying graph $\mathcal{G}$ that are stated as Theorem \ref{main2}. In Section \ref{sec3}, we then explain how these can be used to obtain the scaling limit for the random walk $X$ of Theorem \ref{main1}, and also give a proof of Corollary \ref{srwcor}. Throughout the paper, we write $a_{n} = O (b_{n} )$ if $|a_{n}| \le C b_{n} $ for some absolute constant $C > 0$. Also, we write $a_{n} = o (b_{n} )$ if $a_{n} / b_{n} \to 0$. We let $c, C,\dots $ denote arbitrary constants in $(0,\infty)$ that may change from line to line. To simplify notation, we sometimes use a continuous variable, $x$ say, where a discrete argument is required, with the understanding that it should be treated as $\lfloor x \rfloor$.

\section{Estimates for the underlying graph}\label{sec2}

We establish the parts of Theorem \ref{main2} concerning the measure, the metrics and cut times separately.

\subsection{Volume scaling}

Towards proving the volume asymptotics of Theorem \ref{main2}(a), similarly to \eqref{vdef} and \eqref{edef}, for each for $0\leq m<n$, we define a subgraph ${\cal G}_{m,n} = (V ({\cal G}_{m,n} ), E({\cal G}_{m,n} ) )$ of $\mathcal{G}$ by setting
\begin{equation}\label{traceab}
V ( {\cal G}_{m,n} ) := \left\{ S_k:\: m\leq k\leq n \right\},\qquad E ({\cal G}_{m,n} ) := \left\{ \{S_k, S_{k+1} \}:\: m \le k \le n-1 \right\}.
\end{equation}
We moreover let ${\cal G}_{n}:={\cal G}_{0,n}$.

\begin{proof}[Proof of Theorem \ref{main2}(a)]
For each $k, n \ge 0$, we introduce a random variable $Y_{k}^{(n)}$ by letting
\[Y_{k}^{(n)} := \left| \left\{e \in E ({\cal G}_{k+1} ):\:S_k \in e  \right\} \right| \times  {\bf 1}_{\{S_k\in S_{[0,n]}\text{ and }S_m \neq S_k \text{ for all } m \ge k+1 \}},\]
where $|A|$ stands for the cardinality of a set $A$. In particular, for $Y_{k}^{(n)}$ to be strictly positive, we require $k$ to be the time of the last visit by $S$ to the vertex $S_k\in S_{[0,n]}$. It follows that
\begin{equation}\label{iikae}
\mu_{{\cal G}} \left(S_{[0,n]}\right) = \sum_{k= 0}^{\infty} Y_{k}^{(n)}.
\end{equation}
To see this, for each $x \in S_{[0,n]}$, we set $\sigma_{x} = \max \{ k \ge 0:\:S_k = x \}$ to be the last time that the simple random walk $S$ hits $x$. Since $Y_{k}^{(n)}>0$ if and only if $k = \sigma_{x}$ for some $x \in S_{[0,n]}$, it follows that
\[ \sum_{k= 0}^{\infty} Y_{k}^{(n)} = \sum_{x \in S_{[0,n]}} Y_{\sigma_{x}}^{(n)} =  \sum_{x \in  S_{[0,n]}}  \left| \left\{e \in E ({\cal G}_{\sigma_{x}+1} ):\:x \in e  \right\} \right| = \mu_{{\cal G}} \left( S_{[0,n]}\right),\]
which gives \eqref{iikae}.

Next we will compare the volume of $S_{[0,n]}$ with $\sum_{k= 0}^{n} Y_{k}^{(n)}$. To do this, we let $I_{k}$ be the indicator function of the event that $k$ is a cut time, i.e.,
\begin{equation}\label{ikdef}
I_{k}: = {\bf 1}_{\{k\in\mathcal{T}\}},
\end{equation}
where $\mathcal{T}$ was defined at \eqref{cuttimes}. Also, we set
\[F_{n} := \left\{ I_{k} = 0 \text{ for all } k \in [n, n + n (\log n)^{-6} ] \right\},\]
for the event that there are no cut times in $ [n, n + n (\log n)^{-6} ]$. It follows from \cite[Lemma 7.7.4]{Lawb} that
\begin{equation}\label{law-old}
\mathbf{P} (F_{n}) \le \frac{C\log \log n}{\log n},
\end{equation}
where $C\in(0,\infty)$ is a constant. Suppose that there exists a cut time in $ [n, n + n (\log n)^{-6} ]$. Then we have $Y_{k}^{(n)} = 0$ for all $k \ge n + n ( \log n )^{-6}$. Since $Y_{k}^{(n) } \le 8$, it follows that on the event $F_{n}^{c}$,
\[\left| \sum_{k= 0}^{\infty} Y_{k}^{(n)} - \sum_{k= 0}^{n} Y_{k}^{(n)} \right| = \sum_{k= n+1}^{n + n ( \log n)^{-6}} Y_{k}^{(n)} \le 8 n (\log n)^{-6}.\]
Combining this with \eqref{iikae} and \eqref{law-old}, we find that
\begin{equation}\label{darui}
\mathbf{P} \left( \left|\mu_{{\cal G}} \left(S_{[0,n]} \right) - \sum_{k= 0}^{n} Y_{k}^{(n)} \right| > 8 n (\log n)^{-6} \right)  \le \frac{C\log \log n}{\log n},
\end{equation}
and so, in order to show the desired convergence in $\mathbf{P}$-probability of $n^{-1}\mu_{{\cal G}}(S_{[0,n]})$, it will suffice to prove that, as $n \to \infty$,
\begin{equation}\label{ylim}
\frac{  \sum_{k= 0}^{n} Y_{k}^{(n)}}{n} \to \lambda
\end{equation}
in $\mathbf{P}$-probability for some deterministic constant $\lambda \in (0, \infty)$.

We now derive an upper bound on the variance of $\sum_{k= 0}^{n} Y_{k}^{(n)}$. To do this, take $0 \le k < l \le n$ with $l-k \ge 3 \sqrt{n}$. We want to control the dependence between $Y_{k}^{(n)}$ and $Y_{l}^{(n)}$. Define an event $G$ by setting
\[G := \left\{ S_k \notin S _{[k + \sqrt{n}, \infty )},\: S_l \notin S_{[0, l- \sqrt{n} ]} \right\}.\]
We also let
\begin{align*}
Z_{k}& := \left| \left\{ e \in E ({\cal G}_{k+1} ):\: S_k \in e  \right\} \right| \times  {\bf 1}_{\{S_m \neq S_k\text{ for all } k+1 \le m \le k+\sqrt{n} \}}, \\
W_{l}& := \left| \left\{ e \in E ({\cal G}_{l-\sqrt{n}, l+1} ):\: S_l \in e  \right\} \right|  \times  {\bf 1}_{\{S_m \neq S_l \text{ for all }  m \ge l \}}.
\end{align*}
The condition $l-k \ge 3 \sqrt{n}$ ensures that $Z_{k}$ and $W_{l}$ are independent. Also, we note that on the event $G$, it holds that $Y_{k}^{(n)} = Z_{k}$ and $Y_{l}^{(n)} = W_{l}$. The local central limit theorem (see \cite[Theorem 1.2.1]{Lawb}) implies that
\begin{align*}
\mathbf{P} (G^{c} ) & \le \mathbf{P} \left( S_k \in S _{[k + \sqrt{n}, \infty )} \right) + \mathbf{P} \left( S_l \in S_{[0, l- \sqrt{n} ]} \right) \\
&\le 2 \mathbf{P} \left( S_{0} \in S_{[\sqrt{n}, \infty )} \right) \\
&\leq 2 \sum_{k = \sqrt{n} }^{\infty} \mathbf{P}  \left( S_{k} = S_{0} \right) \\
& \le C \sum_{k = \sqrt{n} }^{\infty} k^{-2} \\
&\le  \frac{C}{\sqrt{n}},
\end{align*}
where constants change line to line. Therefore we have
\[\left| \mathbf{E} \left( Y_{k}^{(n) } \right) -\mathbf{ E} \left(Z_{k} \right) \right| \le  \frac{C}{\sqrt{n}},\qquad\left| \mathbf{E} \left( Y_{l}^{(n) } \right) -\mathbf{ E} \left(W_l\right) \right| \le  \frac{C}{\sqrt{n}}.\]
Since $Z_{k}$ and $W_{l}$ are independent, we consequently obtain that: for $0 \le k < l \le n$ with $l-k \ge 3 \sqrt{n}$,
\begin{align*}
\lefteqn{\mathbf{E} \left( \left( Y_{k}^{(n) } - \mathbf{E} ( Y_{k}^{(n) } ) \right) \left( Y_{l}^{(n) } - \mathbf{E} ( Y_{l}^{(n) } ) \right) \right)} \\
&\leq \mathbf{E} \left( \left( Y_{k}^{(n) } - \mathbf{E} ( Y_{k}^{(n) } ) \right) \left( Y_{l}^{(n) } - \mathbf{E} ( Y_{l}^{(n) } ) \right) \mathbf{1}_{G}\right)+C\mathbf{P}(G^c)\\
&\leq  \mathbf{E} \left( \left( Z_k - \mathbf{E} ( Y_{k}^{(n) } ) \right) \left( W_l - \mathbf{E} ( Y_{l}^{(n) } ) \right) \right)+C\mathbf{P}(G^c)\\
&=\mathbf{E} \left( Z_k - \mathbf{E} ( Y_{k}^{(n) } ) \right) \mathbf{E}\left( W_l - \mathbf{E} ( Y_{l}^{(n) } ) \right) +C\mathbf{P}(G^c)\\
&\leq \frac{C}{\sqrt{n}},
\end{align*}
and it readily follows from this that
\[\text{Var} \left( \sum_{k= 0}^{n} Y_{k}^{(n)} \right) \le C n^{\frac{3}{2}}.\]
In particular, this bound yields
\begin{equation}\label{rrr}
\mathbf{P}\left(\left|\sum_{k= 0}^{n} Y_{k}^{(n)}-\mathbf{E}\left( \sum_{k= 0}^{n} Y_{k}^{(n)} \right) \right|\geq n^{7/8}\right)\leq \frac{C}{n^{1/4}}.
\end{equation}

The conclusion of the previous paragraph means that we have reduced the problem of establishing \eqref{ylim} to showing the asymptotic linearity of the deterministic sequence $\mathbf{E}( \sum_{k= 0}^{n} Y_{k}^{(n)})$, and that is what we do now. Let $S^{1}$ and $S^{2}$ be independent simple random walks on $\mathbb{Z}^{4}$ started at the origin. Define a random variable $Y$ by setting
\[Y = \left| \left\{ \{S^{1}_0, S^{1}_1 \}\right\} \cup \left\{ e \in E ({\cal G}^{2} ):\: 0 \in e  \right\} \right|  \times  {\bf 1}_{\{  S^{1}_k \neq 0 \text{ for all } k \ge 1 \}},\]
where the graph ${\cal G}^{2} = (V ({\cal G}^{2} ), E({\cal G}^{2} ) )$ is defined from $S^2$ as at \eqref{vdef} and \eqref{edef}. We will define the $\lambda$ that appears in the statement of the result as
\[\lambda:= \mathbf{E} (Y).\]
Now, take $\sqrt{n} \le k \le n$, and define two events $H$ and $H'$ by setting
\begin{align*}
H &= \left\{ S_m \neq S_k \text{ for all } m \in [0, k- \sqrt{n} ] \cup [k + \sqrt{n}, \infty ) \right\}, \\
H'&= \left\{ S^{1}_m \neq 0 \text{ and } S^{2}_m \neq 0 \text{ for all } m \ge \sqrt{n} \right\}.
\end{align*}
By again applying the local central limit theorem (see \cite[Theorem 1.2.1]{Lawb}), we see that
\begin{align*}
\mathbf{P}(H^{c}) &\le 2 \mathbf{P} \left( S_{0} \in S_{[\sqrt{n}, \infty )} \right) \leq 2 \sum_{k = \sqrt{n} }^{\infty} \mathbf{P}  \left( S_{k} = S_{0} \right)  \le C \sum_{k = \sqrt{n} }^{\infty} k^{-2} \le  \frac{C}{\sqrt{n}}, \\
\mathbf{P}\left( (H')^{c} \right)& \le 2 \mathbf{P} \left( 0 \in S^{1} [\sqrt{n}, \infty ) \right) \le  \frac{C}{\sqrt{n}}.
\end{align*}
With this in mind, we define
\begin{align*}
&U_{k} = \left| \left\{ e \in E ({\cal G}_{k-\sqrt{n}, k+1} ):\: S_k \in e  \right\} \right| \times  {\bf 1}_{\{   S_m \neq S_k \text{ for all } m \in [k+1, k + \sqrt{n}] \}},\\
&U = \left| \left\{\{ S^{1} (0), S^{1} (1) \}\right\} \cup \left\{ e \in E ({\cal G}^{2}_{\sqrt{n}} ) :\: 0 \in e  \right\} \right|  {\bf 1}_{\{   S^{1}_m \neq 0 \text{ for all } m \in [1,  \sqrt{n}]\}},
\end{align*}
where the graph ${\cal G}^{2}_{m} = (V ({\cal G}^{2}_{m} ), E({\cal G}^{2}_{m} ) )$ is defined from $S^2$ similarly to \eqref{traceab}. For $\sqrt{n} \le k \le n$, on the event $H$ (respectively $H'$), we have $Y^{(n)}_{k} = U_{k}$ (respectively $Y= U$). Thus the translation invariance and time-reversibility of the simple random walk yields that, for $\sqrt{n} \le k \le n$,
\begin{align*}
\mathbf{E} \left(Y_{k}^{(n)} \right) &= \mathbf{E} \left(Y_{k}^{(n)} \mathbf{1}_H \right) +O(n^{-1/2})\\
 &=\mathbf{E} \left(U_k \mathbf{1}_H \right) +O(n^{-1/2})\\
 &= \mathbf{E} \left(U_{k} \right) +O(n^{-1/2})\\
&= \mathbf{E} (U) +O(n^{-1/2})\\
& =\mathbf{E} \left(U\mathbf{1}_{H'}\right)  +O(n^{-1/2})\\
& = \mathbf{E} \left(Y\mathbf{1}_{H'}\right) +O(n^{-1/2})\\
&= \lambda +O(n^{-1/2}).
\end{align*}
This gives
\[\mathbf{E} \left( \sum_{k= 0}^{n} Y_{k}^{(n)} \right) = \lambda n + O (\sqrt{n} ).\]
Recalling \eqref{darui} and \eqref{rrr}, the result at \eqref{ylim} follows.

Finally, since $\mu_\mathcal{G}(S_{[0,n]})$ is clearly monotone in $n$, extending from the above convergence to the functional statement of Theorem \ref{main2}(a) is elementary.
\end{proof}

\subsection{Resistance and distance scaling}

The essential work for the resistance and distance asymptotics of Theorem \ref{main2}(b),(c) was completed in \cite{S}. In particular the following result concerning the effective resistance was established. We simply highlight how this, and the techniques used to prove it, can be adapted to our current purposes.

\begin{lem}[{\cite[Theorems 1.2.1 and 1.2.3]{S}}]\label{rlem} For some deterministic slowly-varying function $\tilde{\psi}:\mathbb{N}\rightarrow \mathbb{R}$ satisfying $\tilde{\psi}(n)=(\log n)^{-\frac{1}{2}+o(1)}$,
\[\frac{R_\mathcal{G}(0,S_{n})}{n\tilde{\psi}(n)}\rightarrow 1\]
in $\mathbf{P}$-probability.
\end{lem}

The corresponding result for the graph distance is as follows.

\begin{lem}\label{dlem} For some deterministic slowly-varying function $\phi:\mathbb{N}\rightarrow \mathbb{R}$ satisfying $\phi(n)=(\log n)^{-\frac{1}{2}+o(1)}$,
\[\frac{d_\mathcal{G}(0,S_{n})}{n\phi(n)}\rightarrow 1\]
in $\mathbf{P}$-probability.
\end{lem}
\begin{proof} The properties of the effective resistance used to prove Lemma \ref{rlem} were simply that:
\begin{itemize}
\item $R_{G} (\cdot, \cdot)$ is a metric on $G$,
\item if $0\leq m<k<n$ are such that $S_{[m, k]} \cap S_{[k+1, n]} = \emptyset$, then
\[R_{{\cal G}_{m,n}} \left( S_m, S_n \right) = R_{{\cal G}_{m,k}}\left( S_m, S_k \right) + R_{{\cal G}_{k,n}}\left( S_k, S_n \right),\]
where $R_{{\cal G}_{m,n}}$ is the effective resistance metric for the graph ${{\cal G}_{m,n}}$, as defined at \eqref{traceab} (with unit resistors placed along edges).
\end{itemize}
These two properties are clearly satisfied when we replace the effective resistance by the shortest path graph distance. So we have the same results for the graph distance. The easy details are left to the reader.
\end{proof}

As we hinted in the introduction, the functional convergence statements of Theorem \ref{main2}(b), (c)
%Can remove space if not needed in formatting.
do not immediately follow from Lemmas \ref{rlem} and \ref{dlem} due to the non-monotonicity of $R_\mathcal{G}(0,S_{n})$ and $d_\mathcal{G}(0,S_{n})$. To deal with this issue, we will apply the following observation about the gaps between cut times.

\begin{lem}\label{tgap} Define ${\cal I}_{n} = \{ i \ge 2 \ | \ T_{i} \le n  \}$ and $T_{(n)}:=T_{\sup{\mathcal{I}_n}}$. It then holds that, as $n\to\infty$,
\[\frac{\max_{i\in {\cal I}_n}\left(T_{i}-T_{i-1}\right)+(n-T_{(n)})}{n(\log n)^{-7/12}}\to 0\]
in $\mathbf{P}$-probability.
\end{lem}
\begin{proof}
We set $u_{n} = n (\log n)^{-\frac{2}{3}}$, $r_{n} = n (\log n)^{-6}$ and $N = \lfloor \frac{n}{u_{n}} \rfloor$. For each $1 \le i \le N+1$, we define the event $Z_{i}$ by
\[Z_{i} = \left\{ {\cal T} \cap [(i-1) u_{n}, (i-1)u_{n} + r_{n} ] \neq \emptyset, \: {\cal T} \cap [i u_{n} - r_{n}, i u_{n} ] \neq \emptyset \right\},\]
where ${\cal T}$ is the set of cut times, as defined at \eqref{cuttimes}. It then follows from \cite[Lemma 7.7.4]{Lawb} that
\[\mathbf{P} (Z_{i} ) \ge 1 - \frac{C \log \log n}{\log n}\]
for all $1 \le i \le N+1$. Thus, since $N \le C (\log n)^{\frac{2}{3}}$, we have that
\[\mathbf{P} \left( \bigcap_{i=1}^{N+1} Z_{i} \right) \ge 1 - \frac{C \log \log n}{(\log n)^{\frac{1}{3}}}.\]
Clearly, on the event $\cap_{i=1}^{N+1} Z_{i}$, it holds that $\max_{i\in {\cal I}_n}(T_{i}-T_{i-1})\leq u_n$ and also $n-T_{(n)}\leq u_n$, and so the result follows.
\end{proof}

The above lemma readily allows us to compare the effective resistance process $R_\mathcal{G}(0,S_{n})$ with its past maximum, $\max_{m\leq n}R_\mathcal{G}(0,S_{n})$, and similarly for the graph distance.

\begin{lem}\label{request}
It holds that, as $n \to \infty$,
\[\frac{\max_{m\leq n}\left|\max_{k\leq m}R_\mathcal{G}(0,S_k)-R_\mathcal{G}(0,S_m)\right|}{n(\log n)^{-7/12}}\rightarrow 0\]
in $\mathbf{P}$-probability. The same claim also holds when $R_\mathcal{G}$ is replaced by $d_\mathcal{G}$.
\end{lem}
\begin{proof} If $m\leq T_1$, it holds that
\[\left|\max_{k\leq m}R_\mathcal{G}(0,S_k)-R_\mathcal{G}(0,S_m)\right|\leq T_1.\]
If $m\in[T_{i-1},T_i]$ for some $i\in\mathcal{I}_n$, where $\mathcal{I}_n$ is defined as in the statement of Lemma \ref{tgap}, then
\[\left|\max_{k\leq m}R_\mathcal{G}(0,S_k)-R_\mathcal{G}(0,S_m)\right|=\left|\max_{k\in[T_{i-1},m]}R_\mathcal{G}(S_{T_{i-1}},S_k)-R_\mathcal{G}(S_{T_{i-1}},S_m)\right|\leq T_i-T_{i-1}.\]
If $m\in [T_{(n)},n]$, then
\[\left|\max_{k\leq m}R_\mathcal{G}(0,S_k)-R_\mathcal{G}(0,S_m)\right|=\left|\max_{k\in[T_{(n)},n]}R_\mathcal{G}(S_{T_{(n)}},S_k)-R_\mathcal{G}(S_{T_{(n)}},S_n)\right|\leq n-T_{(n)}.\]
We thus obtain the first claim of the lemma by applying Lemma \ref{tgap} (and the fact that $T_1$ is $\mathbf{P}$-a.s.\ finite). The second assertion can be proved similarly.
\end{proof}

\begin{proof}[Proof of Theorem \ref{main2}(b),(c)] From Lemmas \ref{rlem} and \ref{request}, we find that
\[\frac{\max_{m\leq n}R_\mathcal{G}(0,S_{m})}{n\tilde{\psi}(n)}\rightarrow 1\]
in $\mathbf{P}$-probability. Since $\max_{m\leq n}R_\mathcal{G}(0,S_{m})$ is monotone, it follows by an elementary argument that
\[\left(\frac{\max_{m\leq nt}R_\mathcal{G}(0,S_{m})}{n\tilde{\psi}(n)}\right)_{t\in[0,T]}\rightarrow \left(t\right)_{t\in[0,T]}\]
in $\mathbf{P}$-probability. Combining this with Lemma \ref{request}, we obtain Theorem \ref{main2}(b). The proof of Theorem \ref{main2}(c) is essentially the same.
\end{proof}

\subsection{Cut-time scaling}

Recall the definition of a cut time from \eqref{cuttimes} and the random variables $(I_{k})_{k\geq 0}$ from \eqref{ikdef}. Moreover, define
\[N_{n} = \sum_{k=0}^{n} I_{k}\]
to be the number of cut times in $[0, n]$. On \cite[page 3]{cut}, it is stated that there exists a deterministic constant $\alpha\in(0,\infty)$ such that, in $\mathbf{P}$-probability, as $n \to \infty$,
\begin{equation}\label{no-proof}
\frac{N_{n}}{n (\log n)^{-\frac{1}{2}}} \to \alpha.
\end{equation}
From this and the monotonicity of the sequence $(T_n)_{n\geq 1}$, Theorem \ref{main2}(d) readily follows with $\tau=\alpha^{-1}$. As we noted in Remark \ref{back}, it is claimed in \cite{cut} that the limit at \eqref{no-proof} can be proved by using the methods developed in \cite[Chapter 7]{Lawb}. However, for the sake of the reader, we will give a proof of \eqref{no-proof} here.

\begin{proof}[Proof of Theorem \ref{main2}(d)]
As noted above, it will suffice to prove \eqref{no-proof}. Let $S^{1} = ( S^{1}_{k} )_{k \ge 0} $ and $S^{2} = ( S^{2}_{k} )_{k \ge 0} $ be two independent simple random walks on $\mathbb{Z}^{4}$ started at the origin. We write $\xi_{n}^{1}$ for the first time that $S^{1}$ exits $\{ x \in \mathbb{Z}^{4} :\:|x| \le n \}$, that is, a ball of radius $n$ with respect to the Euclidean distance centered at the origin. It is then proved in \cite[Theorem 5.2]{SR} that there exists a constant $\beta \in(0,\infty)$ such that
\begin{equation}\label{law-sr}
\lim_{n \to \infty} (\log n )^{\frac{1}{2}}   \mathbf{P} \left( S^{1}_{[1, \xi_{n}^{1} ]} \cap S^{2}_{[0, \infty )} = \emptyset \right) = \beta,
\end{equation}
where we let $S^{1}_{[a, b]} = \{ S^{1}_{k} :\: a \le k \le b \}$ for $0 \le a \le b < \infty$, and $S^{2}_{[0, \infty)} = \{ S^{2}_{k}:\: k \ge 0 \}$.
Moreover, it is known that
\begin{equation}\label{large-dev}
\mathbf{P} \left( n^{2} (\log n)^{-1} \le \xi_{n}^{1} \le n^{2} \log n \right) \ge 1 - C n^{-c},
\end{equation}
for some constants $c, C \in (0, \infty)$, see \cite[Proposition 2.4.5]{LawLim}. Combining \eqref{law-sr} and \eqref{large-dev}, we see that
\begin{align*}
\mathbf{P} \left( S^{1}_{[1, n^{2} \log n ]} \cap S^{2}_{[0, \infty )} = \emptyset \right) &\le \mathbf{P} \left( S^{1}_{[1, \xi_{n}^{1} ]} \cap S^{2}_{[0, \infty )} = \emptyset \right) + C n^{-c}  = \left( \beta + o (1) \right) (\log n)^{-\frac{1}{2}},  \\
\mathbf{P} \left( S^{1}_{[1, n^{2} (\log n)^{-1} ]} \cap S^{2}_{[0, \infty )} = \emptyset \right) &\ge \mathbf{P} \left( S^{1}_{[1, \xi_{n}^{1} ]} \cap S^{2}_{[0, \infty )} = \emptyset \right) - C n^{-c} = \left( \beta + o (1) \right) (\log n)^{-\frac{1}{2}}.
\end{align*}
Reparameterising and letting $\alpha = \sqrt{2} \beta$, we find that
\[\lim_{n \to \infty} (\log n)^{\frac{1}{2}} \mathbf{P} \left( S^{1}_{[1, n ]} \cap S^{2}_{[0, \infty )} = \emptyset \right) = \alpha,\]
and thus that
\begin{equation}\label{exponent-int}
\lim_{n \to \infty} (\log n)^{\frac{1}{2}} \mathbf{E} (I_{n} )= \alpha.
\end{equation}
Consequently, it follows that
\begin{equation}\label{nnas}
\lim_{n \to \infty} \frac{\mathbf{E}(N_{n})}{n (\log n)^{-\frac{1}{2}}} = \alpha.
\end{equation}
Thus to complete the proof, it will be enough to establish a suitable bound on the variance of $N_n$ for large $n$.

Take $n_{0} \ge 1$ so that: for all $n_{0} \le m \le n$,
\begin{equation}\label{large-n}
b_{m} \le b_{n},
\end{equation}
where $b_{n} = n (\log n)^{-9}$. Throughout the proof here, we will assume $b_n\ge n_{0}$. For $b_{n} \le k \le n$, we set
\[I_{k,n} = {\bf 1}_{\{ S_{[k- b_{n}, k]} \cap S_{[k+1, k+ b_{n} ]} = \emptyset \}}.\]
Also, we write
\begin{align*}
F_{k}& := \{ k \text{ is a cut time} \}\equiv\{I_k=1\},\qquad\\
F_{k, n}& :=  \left\{ S_{[k- b_{n}, k]} \cap S_{[k+1, k+ b_{n} ]} = \emptyset \right\}\equiv\{I_{k,n}=1\}.
\end{align*}
We note that $F_{k,n}$ and $F_{l,n}$ are independent if $b_{n} \le k < l $ and $l-k > 2 b_{n}$. In order to estimate the variance of $N_{n}$, we consider $b_{n} \le k < l \le n $ with $l-k > 2 b_{n}$. Noting that $ F_{k} \cap F_{l}\subseteq F_{k,n} \cap F_{l, n}$, we have that
\begin{align*}
\left| \mathbf{E} (I_{k} I_{l} ) - \mathbf{E} (I_{k, n} I_{l,n} ) \right|& = \mathbf{P}( F_{k,n} \cap F_{l, n} ) - \mathbf{P} ( F_{k} \cap F_{l} )  \\
&\le \mathbf{P}( F_{k,n} \cap F_{l, n}  \cap F_{k}^{c} ) +   \mathbf{P}( F_{k,n} \cap F_{l, n}  \cap F_{l}^{c} )\\
&\le \mathbf{P}( F_{k,n}  \cap F_{k}^{c} ) +   \mathbf{P}( F_{l, n}  \cap F_{l}^{c} ).
\end{align*}
We will estimate the probabilities on the right-hand side. It is shown in  \cite[Lemma 7.7.3]{Lawb} that
\[ \mathbf{P} ( F_{n, n} ) =  \mathbf{P}( F_{n} ) \left( 1 + O \left( \frac{\log \log n}{ \log n} \right) \right).\]
Since $ n_{0} \le b_{n} \le k \le n$, we have $b_{k} \le b_{n}$ by \eqref{large-n}. From this, it follows that $F_{k} \subseteq F_{k, n} \subseteq F_{k,k}$.
Thus we have
\begin{equation}\label{atode}
\mathbf{P} ( F_{k} ) \le \mathbf{P} (F_{k, n} ) \le \mathbf{P} (F_{k,k})  =  \mathbf{P}( F_{k} ) \left( 1 + O \left( \frac{\log \log k}{ \log k} \right) \right) =  \mathbf{P}( F_{k} ) \left( 1 + O \left( \frac{\log \log n}{ \log n} \right) \right),
\end{equation}
which implies in turn that
\[\mathbf{P}( F_{k,n}  \cap F_{k}^{c} ) =  \mathbf{P} (F_{k, n} ) - \mathbf{P} ( F_{k} ) = \mathbf{P} ( F_{k} ) O \left( \frac{\log \log n}{ \log n} \right).\]
Moreover, by \eqref{exponent-int}, it holds that $\mathbf{P} ( F_{k} ) \le C (\log n )^{-\frac{1}{2}}$, and so
\[\mathbf{P}( F_{k,n}  \cap F_{k}^{c} ) = O \left( \log \log n / (\log n)^{\frac{3}{2}} \right).\]
A similar estimate applies to $\mathbf{P}( F_{l,n}  \cap F_{l}^{c} )$, and thus we conclude that
\[\left| \mathbf{E}(I_{k} I_{l} ) - \mathbf{E}(I_{k, n} I_{l,n} ) \right| \le C ( \log n)^{- \frac{3}{2}} (\log \log n),\]
for $b_{n} \le k < l \le n $ with $l-k  > 2 b_{n}$.

We are now in a position to complete the proof. It follows from \eqref{atode} that, for $b_{n} \le k \le n$,
\[\mathbf{E}(I_{k, n} ) = \mathbf{E}(I_{k} ) \left( 1 + O \left( \frac{\log \log n}{\log n} \right) \right).\]
Also, we recall that $I_{k,n}$ and $I_{l, n}$ are independent for $b_{n} \le k, l \le n$ with $|l-k| > 2 b_{n}$.
Consequently, writing $\sum_{\ast}$ for the sum over integers $k$ and $l$ satisfying $b_{n} \le k, l \le n$ with $|l-k| > 2 b_{n}$,
\begin{align*}
\text{Var} (N_{n} )& = \sum_{0 \le k, l \le n} \left( \mathbf{E}(I_{k} I_{l} ) - \mathbf{E}(I_{k} ) \mathbf{E}(I_{l} ) \right)  \\
&= \sum\nolimits_{\ast} \left( \mathbf{E}(I_{k} I_{l} ) - \mathbf{E}(I_{k} ) \mathbf{E}(I_{l} ) \right) + O (n b_{n})  \\
&= \sum\nolimits_{\ast} \left( \mathbf{E}(I_{k, n} I_{l, n} ) - \mathbf{E}(I_{k} ) \mathbf{E}(I_{l} ) \right) + O \left( n^{2}  ( \log n)^{- \frac{3}{2}} (\log \log n) \right) \\
&= \sum\nolimits_{\ast}  \mathbf{E}(I_{k} ) \mathbf{E}(I_{l} ) O \left( \frac{\log \log n}{\log n} \right)  + O \left( n^{2}  ( \log n)^{- \frac{3}{2}} (\log \log n) \right) \\
&= \mathbf{E}( N_{n} )^{2} O \left(   ( \log n)^{- \frac{1}{2}} (\log \log n) \right),
\end{align*}
where for the final line we recall \eqref{nnas}. This implies that $N_{n} / \mathbf{E}( N_{n} )$ converges to $1$ in $\mathbf{P}$-probability as $n \to \infty$. Combining this with \eqref{nnas}, we get the limit at \eqref{no-proof}, as required.
\end{proof}

\section{Simple random walk scaling limit}\label{sec3}

As briefly outlined in the introduction, the way in which we will establish the random walk scaling limits of Theorem \ref{main1} is via the resistance form approach developed in \cite{Croyres,CHK}. In particular, we will apply \cite[Theorem 7.2]{Croyres}, which requires us to check a certain convergence result for compact metric spaces equipped with measures, marked points and continuous maps (see Proposition \ref{p1} below), as well as a non-explosion condition (see Proposition \ref{p2}). To this end, we start by  introducing the framework that will enable us to check the conditions of \cite[Theorem 7.2]{Croyres}. Once these have been verified, we recall the necessary resistance form background for deriving our main conclusions for the random walk on the range of random walk in four dimensions that appear as Theorem \ref{main1}. To conclude the section, we prove Corollary \ref{srwcor}.

Broadly following the notation of \cite{Croyres}, we consider a space $\mathbb{K}^*_c$ of elements of the form $(K, d_K, \mu, \rho,f)$, where we assume: $(K, d_K)$ is a compact metric space, $\mu$ is a finite Borel measure on $K$, $\rho$ is a marked point of $K$, and $f$ is a continuous map from $K$ into some image space $(M,d_M)$, which is a fixed complete, separable metric space. (When proving Theorem \ref{main1}(a), we will take $(M,d_M)$ to be $\mathbb{R}$ equipped with the Euclidean distance, and when proving Theorem \ref{main1}(b), we will take $(M,d_M)$ to be $\mathbb{R}^4$, again equipped with the Euclidean distance.) For two elements of $\mathbb{K}^*_c$, we define the distance between them $\Delta^*_c((K, d_K, \mu, \rho, f), (K', d_{K'}, \mu', \rho', f'))$ to be equal to
\begin{equation}\label{deltacdef}
\inf_{\substack{(Z,d_Z),\pi,\pi',\mathcal{C}:\\(\rho,\rho')\in\mathcal{C}}}\left\{d_Z^P\left(\mu\circ\pi^{-1},\mu'\circ{\pi'}^{-1}\right)+ \sup_{(x,x')\in\mathcal{C}}\left(d_Z(\pi(x),\pi'(x'))+d_M(f(x),f'(x'))\right)\right\},
\end{equation}
where the infimum is taken over all metric spaces $(Z, d_Z)$, isometric embeddings  $\pi: (K, d_K)\rightarrow (Z,d_Z)$, $\pi': (K', d_{K'}) \rightarrow(Z,d_Z)$, and correspondences $\mathcal{C}$ between $K$ and $K'$. Note that, by a correspondence $\mathcal{C}$ between $K$ and $K'$, we mean a subset of $K\times  K'$ such that for every $x\in K$ there exists at least one $x'\in K'$ such that $(x, x')\in\mathcal{C}$, and conversely, for every $x' \in K'$ there exists at least one $x\in K$ such that $(x, x')\in\mathcal{C}$. To make sense of the expression at \eqref{deltacdef}, we further define $d_Z^P$ to be the Prohorov distance between finite Borel measures on $(Z,d_Z)$. We note that, up to trivial equivalences, it is possible to check that $(\mathbb{K}^*_c,\Delta^*_c)$ is a separable metric space.

For checking Theorem \ref{main1}(a), the (random) elements of $\mathbb{K}^*_c$ that will be of interest to us will be bounded restrictions of
\[\mathcal{X}_n:=\left(V(\mathcal{G}),\frac{R_\mathcal{G}}{n\tilde{\psi}(n)},\frac{\mu_\mathcal{G}}{\lambda n},0,\frac{d_\mathcal{G}(0,\cdot)}{n{\phi}(n)}\right),\]
where $\lambda$, $\tilde{\psi}$ and $\phi$ are given by Theorem \ref{main2}, and
\[\mathcal{X}:=\left([0,\infty),d_E,\mathcal{L},0,{I}\right),\]
where $d_E$ is the Euclidean distance on $[0,\infty)$, $\mathcal{L}$ is Lebesgue measure on $[0,\infty)$, and $I$ is the identity map on $[0,\infty)$. In particular, for $r\in(0,\infty)$, let
\[\mathcal{X}^{(r)}_n:=\left(B_\mathcal{G}(0,rn\tilde{\psi}(n)),\frac{R_\mathcal{G}}{n\tilde{\psi}(n)},\frac{\mu_\mathcal{G}}{\lambda n},0,\frac{d_\mathcal{G}(0,\cdot)}{n{\phi}(n)}\right),\]
where $B_\mathcal{G}(x,r):=\{y\in V(\mathcal{G}):\:R_\mathcal{G}(x,y)< r\}$ is the (open) resistance ball in $\mathcal{G}$ with centre $x$ and radius $r$, and
\[\mathcal{X}^{(r)}:=\left([0,r],d_E^{(r)},\mathcal{L}^{(r)},0,{I}^{(r)}\right),\]
where $d_E^{(r)}$, $\mathcal{L}^{(r)}$ and ${I}^{(r)}$ are the natural restrictions of $d_E$, $\mathcal{L}$ and $I$ to $[0,r]$. Similarly, for checking Theorem \ref{main2}(b), we introduce
\[\tilde{\mathcal{X}}_n:=\left(V(\mathcal{G}),\frac{R_\mathcal{G}}{n\tilde{\psi}(n)},\frac{\mu_\mathcal{G}}{\lambda n},0,\frac{I_{\mathcal{G}}}{n^{1/2}}\right),\]
where $I_{\mathcal{G}}$ is the identity map on $\mathbb{R}^4$ restricted to $V(\mathcal{G})$, and
\[\tilde{\mathcal{X}}:=\left([0,\infty),d_E,\mathcal{L},0,W_{\cdot}\right),\]
where $W$ is standard Brownian motion on $\mathbb{R}^4$. The bounded restrictions of these are then given by
\[\tilde{\mathcal{X}}^{(r)}_n:=\left(B_\mathcal{G}(0,rn\tilde{\psi}(n)),\frac{R_\mathcal{G}}{n\tilde{\psi}(n)},\frac{\mu_\mathcal{G}}{\lambda n},0,\frac{I_{n}^{(r)}}{n^{1/2}}\right),\]
where $I_{n}^{(r)}$ is the identity map on $\mathbb{R}^4$ restricted to $B_\mathcal{G}(0,rn\tilde{\psi}(n))$, and
\[\tilde{\mathcal{X}}^{(r)}:=\left([0,r],d_E^{(r)},\mathcal{L}^{(r)},0,W_{\cdot}^{(r)}\right),\]
where $W^{(r)}:[0,r]\rightarrow\mathbb{R}^4$ is the continuous map given by $x\mapsto W_x$ on the relevant interval. Note that, since $R_\mathcal{G}(0,S_{T_m})\geq m-1$ and $T_m<\infty$ for all $m$, the set $B_\mathcal{G}(0,rn\tilde{\psi}(n))$ is finite, and therefore $R_\mathcal{G}$-compact, $\mathbf{P}$-a.s. In particular, both $\mathcal{X}^{(r)}_n$ and $\tilde{\mathcal{X}}^{(r)}_n$ are elements of $\mathbb{K}_c^*$, $\mathbf{P}$-a.s. It is clear that $\mathcal{X}^{(r)}$ and $\tilde{\mathcal{X}}^{(r)}$ are as well. The first goal of this section is to prove the following. We note that the proof is similar to that of \cite[Theorem 4.1]{CFJ}, which demonstrated a corresponding convergence result for the resistance metric spaces associated with the one-dimensional Mott random walk.

\begin{prop}\label{p1} (a) For every $r\in(0,\infty)$, as $n\rightarrow\infty$,
\[\mathcal{X}_{n}^{(r)}\rightarrow \mathcal{X}^{(r)}\]
in $\mathbf{P}$-probability in the space $(\mathbb{K}^*_c,\Delta^*_c)$, where $(M,d_M)$ is $\mathbb{R}$ equipped with the Euclidean distance.\\
(b) For every $r\in(0,\infty)$, as $n\rightarrow\infty$,
\[\tilde{\mathcal{X}}_{n}^{(r)}\rightarrow \tilde{\mathcal{X}}^{(r)}\]
in distribution in the space $(\mathbb{K}^*_c,\Delta^*_c)$, where $(M,d_M)$ is $\mathbb{R}^4$ equipped with the Euclidean distance.
\end{prop}
\begin{proof} First, define
\begin{equation}\label{rnrdef}
R(n,r):=\sum_{i=1}^\infty \mathbf{1}_{\{R_\mathcal{G}(0,S_{T_i})< rn\tilde{\psi}(n)\}}
\end{equation}
to be the number of cut points that fall into $B_\mathcal{G}(0,rn\tilde{\psi}(n))$, and set
\[V_n^{(r)}:=\{x_0,x_1,\dots,x_{R(n,r)}\},\]
where $x_0:=0$ and
\[x_i:=\frac{R_\mathcal{G}\left(0,S_{T_i}\right)}{n\tilde{\psi}(n)}.\]
Note that, by construction $x_0\leq x_1<x_2<\dots< x_{R(n,r)}$ (with equality between $x_0$ and $x_1$ if and only if $T_1=0$). We will approximate $\mathcal{X}_{n}^{(r)}$ by
\[\mathcal{Y}_{n}^{(r)}:=\left(V_n^{(r)},d_E|_{V_n^{(r)}},\mu_n^{(r)},0,I|_{V_{n}^{(r)}}\right),\]
where $\mu_n^{(r)}$ is a measure on $V_n^{(r)}$ given by
\[\mu_n^{(r)}(\{x_i\})=\frac{\mu_\mathcal{G}\left(\{S_{T_{i-1}+1},\dots,S_{T_i}\}\right)}{\lambda n}\]
for $i\geq 2$, $\mu_n^{(r)}(\{x_1\})={\mu_\mathcal{G}(\{S_{0},\dots,S_{T_1}\})}/{\lambda n}$, and $\mu_n^{(r)}(x_0)=0$ if $x_0\neq x_1$. In particular, we claim that, for any $r\in(0,\infty)$,
\begin{equation}\label{claim1}
\Delta_c^*\left(\mathcal{X}_{n}^{(r)},\mathcal{Y}_{n}^{(r)}\right)\rightarrow0
\end{equation}
in $\mathbf{P}$-probability as $n\rightarrow\infty$.

To prove \eqref{claim1}, for each $n\geq 1$, we consider the embedding $\pi_n$ of $V_n^{(r)}$ into $B_\mathcal{G}(0,rn\tilde{\psi}(n))$ given by $x_i\mapsto S_{T_i}$ for $i\geq1$, and $x_0\mapsto 0$. Moreover, we define a correspondence $\mathcal{C}_n$ by pairing $x_0$ with 0, $x_1$ with each other element of $\{S_0,\dots,S_{T_1}\}$, $x_i$ with each element of $\{S_{T_{i-1}+1},\dots,S_{T_i}\}$ for $i=2,\dots,R(n,r)$, and $x_{R(n,r)}$ with each remaining element of $B_\mathcal{G}(0,rn\tilde{\psi}(n))$. For this particular choice of embedding and correspondence, we will estimate each of the terms in the definition of $\Delta_c^*$ given at \eqref{deltacdef}, where $(Z,d_Z)$ is given by $B_\mathcal{G}(0,rn\tilde{\psi}(n))$ equipped with the rescaled resistance metric, i.e.\ $R_\mathcal{G}(\cdot,\cdot)/n\tilde{\psi}(n)$. (Note that $\pi_n$ is indeed an isometry from $(V_n^{(r)},d_E|_{V_n^{(r)}})$ into $(Z,d_Z)$.) Firstly, observe that $\mu_n^{(r)}\circ\pi_n^{-1}$ is the measure on $B_\mathcal{G}(0,rn\tilde{\psi}(n))$ obtained by assigning mass ${\mu_\mathcal{G}(\{S_{0},\dots,S_{T_1}\})}/{\lambda n}$ to $S_{T_1}$, and, for each $i=2,\dots,R(n,r)$, assigning mass $\mu_\mathcal{G}(\{S_{T_{i-1}+1},\dots,S_{T_i}\})/\lambda n$ to $S_{T_i}$ (and assigning mass 0 to all other elements of $B_\mathcal{G}(0,rn\tilde{\psi}(n))$). Hence, applying the definition of the Prohorov metric, it readily follows that
\begin{eqnarray*}
\lefteqn{d_Z^P\left(\mu_n^{(r)}\circ\pi_n^{-1},\frac{\mu_\mathcal{G}(\cdot\cap B_\mathcal{G}(0,rn\tilde{\psi}(n)))}{\lambda n}\right)}\\
&\leq& d_Z^P\left(\mu_n^{(r)}\circ\pi_n^{-1},\frac{\mu_\mathcal{G}(\cdot\cap B_\mathcal{G}(0,\{S_0,S_1,\dots,S_{T_{R(n,r)}}\}))}{\lambda n}\right)+\frac{\mu_\mathcal{G}(\{S_{T_{R(n,r)}+1},\dots,S_{T_{R(n,r)+1}}\})}{\lambda n}\\
&\leq &\max_{m\leq{T_1}}\frac{R_\mathcal{G}(S_m,S_{T_1})}{n\tilde{\psi}(n)}+\max_{i=2,\dots,R(n,r)}\max_{m\in \{T_{i-1}+1,\dots,{T_i}\}}\frac{R_\mathcal{G}(S_m,S_{T_i})}{n\tilde{\psi}(n)}\\
&&\hspace{60pt}+\frac{\mu_\mathcal{G}(\{S_{T_{R(n,r)}+1},\dots,S_{T_{R(n,r)+1}}\})}{\lambda n}.
\end{eqnarray*}
Since $\max_{m\leq{T_1}}R_\mathcal{G}(S_m,S_{T_1})\leq T_1$ is $\mathbf{P}$-a.s.\ a finite random variable, the first term converges to 0, $\mathbf{P}$-a.s.\ As for the second and third terms, for large $n$, the sum of these is bounded above by
\begin{equation}\label{tbound}
\max_{i=2,\dots,R(n,r)+1}\frac{2(T_{i}-T_{i-1})}{n\tilde{\psi}(n)}.
\end{equation}
Now, by Theorem \ref{main2}(b), it holds that, for any $T\in(0,\infty)$,
\[\left(n^{-1}\sup\left\{m:\:R_\mathcal{G}(0,S_m)\leq t n \tilde{\psi}(n)\right\}\right)_{t\in[0,T]}\rightarrow (t)_{t\in[0,T]}\]
in $\mathbf{P}$-probability. Combined with \eqref{no-proof}, it follows that
\begin{equation}\label{e0}
\frac{R(n,r)}{n(\log n)^{-1/2}}\rightarrow \frac{r}{\tau}
\end{equation}
in $\mathbf{P}$-probability. Therefore, again applying \eqref{no-proof}, with probability going to one as $n\rightarrow \infty$,
\begin{equation}\label{e1}
\max_{i=2,\dots,R(n,r)+1}\frac{2(T_{i}-T_{i-1})}{n\tilde{\psi}(n)}\leq \max_{i\in \mathcal{I}_{2nr}}\frac{2(T_{i}-T_{i-1})}{n\tilde{\psi}(n)},
\end{equation}
where the notation $\mathcal{I}_{2nr}$ is defined as in the statement of Lemma \ref{tgap}. Applying the latter result, we see that the right-hand side above converges to zero in $\mathbf{P}$-probability, and thus we have established that,  in $\mathbf{P}$-probability,
\[d_Z^P\left(\mu_n^{(r)}\circ\pi_n^{-1},\frac{\mu_\mathcal{G}(\cdot\cap B_\mathcal{G}(0,rn\tilde{\psi}(n)))}{\lambda n}\right)\rightarrow 0.\]
Secondly,
\begin{eqnarray}
\max_{(x,x')\in\mathcal{C}_n}\frac{R_\mathcal{G}\left(\pi_n(x),x'\right)}{n\tilde{\psi}(n)}&\leq &
\max_{m\leq{T_1}}\frac{R_\mathcal{G}(S_m,S_{T_1})}{n\tilde{\psi}(n)}+\max_{i=2,\dots,R(n,r)}\max_{m\in \{T_{i-1}+1,\dots,{T_i}\}}\frac{R_\mathcal{G}(S_m,S_{T_i})}{n\tilde{\psi}(n)}\nonumber\\
&&\hspace{60pt}+\max_{m\in \{T_{R(n,r)}+1,\dots,{T_{R(n,r)+1}}\}}\frac{R_\mathcal{G}(S_m,S_{T_{R(n,r)}})}{n\tilde{\psi}(n)}.\label{e2}
\end{eqnarray}
Again, the first term converges to zero $\mathbf{P}$-a.s. Moreover, bounding the sum of the second and third terms by the expression at \eqref{tbound} and arguing as at \eqref{e1} shows that they converge to zero in $\mathbf{P}$-probability. Thirdly,
\begin{eqnarray*}
\max_{(x,x')\in\mathcal{C}_n} \left|x-\frac{d_\mathcal{G}(0, x')}{n\phi(n)}\right|&\leq &\max_{i=1,\dots,R(n,r)}\left|\frac{R_\mathcal{G}\left(0,S_{T_i}\right)}{n\tilde{\psi}(n)}-\frac{d_\mathcal{G}\left(0,S_{T_i}\right)}{n{\phi}(n)}\right|\\
&&+\max_{m\leq{T_1}}\frac{d_\mathcal{G}(S_m,S_{T_1})}{n{\phi}(n)}+\max_{i=2,\dots,R(n,r)}\max_{m\in \{T_{i-1}+1,\dots,{T_i}\}}\frac{d_\mathcal{G}(S_m,S_{T_i})}{n{\phi}(n)}\\
&&\hspace{60pt}+\max_{m\in \{T_{R(n,r)}+1,\dots,{T_{R(n,r)+1}}\}}\frac{d_\mathcal{G}(S_m,S_{T_{R(n,r)}})}{n{\phi}(n)}.
\end{eqnarray*}
The second, third and fourth terms can be dealt with exactly as were the terms on the right-hand side of \eqref{e2}. As for the first term, by Theorem \ref{main2}(d) and \eqref{e0}, we have that, with probability going to one as $n\rightarrow \infty$,
\[\max_{i=1,\dots,R(n,r)}\left|\frac{R_\mathcal{G}\left(0,S_{T_i}\right)}{n\tilde{\psi}(n)}-\frac{d_\mathcal{G}\left(0,S_{T_i}\right)}{n{\phi}(n)}\right|\leq
\max_{m=1,\dots,2nr}\left|\frac{R_\mathcal{G}\left(0,S_{m}\right)}{n\tilde{\psi}(n)}-\frac{d_\mathcal{G}\left(0,S_{m}\right)}{n{\phi}(n)}\right|.\]
Hence, from Theorem \ref{main2}(b),(c), this also converges to zero in $\mathbf{P}$-probability. This completes the proof of \eqref{claim1}.

As a consequence of \eqref{claim1}, to establish part (a) of the proposition, it will suffice to check that
\begin{equation}\label{claim2}
\mathcal{Y}_{n}^{(r)}\rightarrow \mathcal{X}^{(r)}
\end{equation}
in $\mathbf{P}$-probability in the space $(\mathbb{K}^*_c,\Delta^*_c)$, where $(M,d_M)$ is $\mathbb{R}$ equipped with the Euclidean distance. For this, we start by noting that $(V_n^{(r)},d_E|_{V_n^{(r)}})$ is isometrically embedded into $([0,r],d_E^{(r)})$ by the identity map. Moreover, similarly to above, it is possible to define a correspondence $\mathcal{C}$ between $V_n^{(r)}$ and $[0,r]$ by pairing $x_0$ with 0, $x_i$ with each element of $(x_{i-1},x_i]$ for $i=1,\dots,R(n,r)$, and $x_{R(n,r)}$ with each remaining element of $[0,r]$. Now, if $a\in(0,r]$ is such that $R_\mathcal{G}(0,S_{T_1})\leq an\tilde{\psi}(n)$, then
\[\mu_n^{(r)}\left([0,a]\right)=\frac{\mu_\mathcal{G}\left(S_{[0,T_{R(n,a)}]}\right)}{\lambda n}.\]
Hence, by applying Theorem \ref{main2}(a),(d) and \eqref{e0}, we find that, for each fixed $a\in(0,r]$, $\mu_n^{(r)}\left([0,a]\right)\rightarrow a$ in $\mathbf{P}$-probability as $n\to\infty$. It readily follows that, writing $d^P$ for the Prohorov distance on $\mathbb{R}$,
\[d^P\left(\mu_n^{(r)},\mathcal{L}^{(r)}\right)\to 0\]
in $\mathbf{P}$-probability as $n\to\infty$. Moreover,
\begin{equation}\label{xnr}
\sup_{(x,x')\in\mathcal{C}}\left(d_E(x,x'))\right)\leq x_1+\sup_{i=2,\dots,R(n,r)}\left(x_{i}-x_{i-1}\right)+r-x_{R(n,r)}.
\end{equation}
Now, $x_1=R_\mathcal{G}(0,S_{T_1})/n\tilde{\psi}(n)\rightarrow 0$, $\mathbf{P}$-a.s., and the sum of the remaining terms can be bounded above by the expression at \eqref{tbound}, and so converges to zero in $\mathbf{P}$-probability. In conclusion, since the maps we consider for both the approximating and limiting spaces are the identity map, we have proved that
\[\Delta_c^*\left(\mathcal{Y}_{n}^{(r)},\mathcal{X}^{(r)}\right)\rightarrow0\]
in $\mathbf{P}$-probability as $n\rightarrow\infty$, which confirms \eqref{claim2}.

The proof of part (b) is similar, though we need to deal with the embedding into $\mathbb{R}^d$. To do this, we will first apply Skorohod coupling to give us a sequence $(S^n)_{n\geq 1}$ of copies of $S$ built on the same probability space as $W$ such that, $\mathbf{P}$-a.s.
\begin{equation}\label{sconv}
\left(n^{-1/2}S^n_{nt}\right)_{t\geq 0}\rightarrow\left(W_t\right)_{t\geq 0},
\end{equation}
uniformly on compact intervals of $t$. We then suppose $\tilde{\mathcal{X}}_n^{(r)}$ is built from $S^n$, and further define
\[\tilde{\mathcal{Y}}_{n}^{(r)}:=\left(V_n^{(r)},d_E|_{V_n^{(r)}},\mu_n^{(r)},0,f_n^{(r)}\right),\]
where the first four components are defined as for $\mathcal{Y}_n^{(r)}$ above (but now from $S^n$), and $f_n^{(r)}$ is the map from $V_n^{(r)}$ into $\mathbb{R}^4$ given by $f_n^{(r)}(x_i)=n^{-1/2}S^n_{T_i}$ for $i=1,\dots,R(n,r)$, and $f_n^{(r)}(x_0)=0$. To show that
\begin{equation}\label{claim3}
\Delta_c^*\left(\tilde{\mathcal{X}}_{n}^{(r)},\tilde{\mathcal{Y}}_{n}^{(r)}\right)\rightarrow0
\end{equation}
in $\mathbf{P}$-probability as $n\rightarrow\infty$, we proceed as in the first part of the proof. We bound the only term not already dealt with as follows:
\begin{eqnarray}
\max_{(x,x')\in\mathcal{C}_n}\left|f_n^{(r)}(x)-\frac{x'}{n^{1/2}}\right|&\leq &
\max_{m\leq{T_1}}\frac{|S_m^n-S_{T_1}^n|}{n^{1/2}}+\max_{i=2,\dots,R(n,r)}\max_{m\in \{T_{i-1}+1,\dots,{T_i}\}}\frac{|S_m^n-S_{T_i}^n|}{n^{1/2}}\nonumber\\
&&\hspace{60pt}+\max_{m\in \{T_{R(n,r)}+1,\dots,T_{R(n,r)+1}\}}\frac{|S_m^n-S_{T_{R(n,r)}}^n|}{n^{1/2}}\nonumber\\
&\leq& 3\max_{\substack{m,k\leq T_{R(n,r)+1}:\\ |m-k|\leq \delta(n,r)}}\frac{|S_m^n-S_k^n|}{n^{1/2}},\label{e22}
\end{eqnarray}
where
\[\delta(n,r):=T_1+\max_{i=2,\dots,R(n,r)+1}(T_i-T_{i-1}).\]
Now, from \eqref{no-proof} and \eqref{e0}, we have that, with $\mathbf{P}$-probability converging to one as $n\to\infty$, $R(n,r)+1\leq N_{2nr}$. It therefore follows from Lemma \ref{tgap} that $n^{-1}\delta(n,r)\rightarrow 0$ in $\mathbf{P}$-probability as $n\rightarrow \infty$. Moreover, from Theorem \ref{main2}(d) and \eqref{e0}, we have that $n^{-1}(T_{R(n,r)+1})\rightarrow r$ in $\mathbf{P}$-probability. Combining these observations with \eqref{sconv} shows that the bound at \eqref{e22} converges to zero in $\mathbf{P}$-probability, and hence we arrive at \eqref{claim3}. Thus it remains to show that
\begin{equation}\label{claim4}
\tilde{\mathcal{Y}}_{n}^{(r)}\rightarrow \tilde{\mathcal{X}}^{(r)}
\end{equation}
in $\mathbf{P}$-probability in the space $(\mathbb{K}^*_c,\Delta^*_c)$, where $(M,d_M)$ is $\mathbb{R}^4$ equipped with the Euclidean distance. The remaining term to deal with is now
\begin{eqnarray*}
\sup_{(x,x')\in\mathcal{C}}\left|f_n^{(r)}(x)-W_{x'}\right|&\leq&
\max_{i=1,\dots,R(n,r)}\sup_{x'\in [x_{i-1},x_i]}\left|W_{x'}-W_{x_i}\right|+\sup_{x'\in[x_{R(n,r)},r]}\left|W_{x'}-W_{x_{R(n,r)}}\right|\\
&&\hspace{60pt}+\max_{i=1,\dots,R(n,r)}\left|W_{x_i}-\frac{S^n_{T_i}}{n^{1/2}}\right|.
\end{eqnarray*}
By \eqref{xnr} and the continuity of $W$, the first two terms converge to zero in $\mathbf{P}$-probability. For the third term, we have
\begin{equation}\label{ttt}
\max_{i=1,\dots,R(n,r)}\left|W_{x_i}-\frac{S^n_{T_i}}{n^{1/2}}\right|\leq\max_{i=1,\dots,R(n,r)}\left|W_{x_i}-\frac{S^n_{nx_i}}{n^{1/2}}\right|+\max_{i=1,\dots,R(n,r)}\left|\frac{S_{nx_i}^n-S^n_{T_i}}{n^{1/2}}\right|.
\end{equation}
Since $x_{R(n,r)}$ converges to $r$ in $\mathbf{P}$-probability, the first term above converges to zero in $\mathbf{P}$-probability by \eqref{sconv}. As for the second term above, we first note that
\[n^{-1}\max_{i=1,\dots,R(n,r)}\left|nx_i-T_i\right|\leq \sup_{t\leq n^{-1}T_{R(n,r)}}\left|\frac{R_\mathcal{G}(0,S_{nt})}{n\tilde{\psi}(n)}-t\right|.\]
Again applying the fact that $n^{-1}T_{R(n,r)}\rightarrow r$ in $\mathbf{P}$-probability, Theorem \ref{main2}(b) gives that the above upper bound converges to zero in $\mathbf{P}$-probability. In combination with \eqref{sconv}, this establishes that the second term at \eqref{ttt} also converges to zero in $\mathbf{P}$-probability. Thus we have completed the proof of \eqref{claim4}.
\end{proof}

The next result gives that the resistance across balls diverges at the same rate as the point-to-point resistance.

\begin{prop}\label{p2} It holds that
\[\lim_{r\rightarrow\infty}\liminf_{n\rightarrow\infty}\mathbf{P}\left(R_\mathcal{G}\left(0,B_\mathcal{G}(0,rn\tilde{\psi}(n))^c\right)\geq \Lambda n\tilde{\psi}(n)\right)=1,\qquad \forall \Lambda\geq 0.\]
\end{prop}
\begin{proof} Using the notation from the proof of Proposition \ref{p1}, we clearly have that
\[R_\mathcal{G}\left(0,B_\mathcal{G}(0,rn\tilde{\psi}(n))^c\right)\geq n\tilde{\psi}(n)x_{R(n,r)}.\]
Hence, since $x_{R(n,r)}\rightarrow r$ in $\mathbf{P}$-probability as $n\rightarrow \infty$ (as follows from the convergence to zero of the expression at \eqref{xnr}), it holds that
\[\liminf_{n\rightarrow\infty}\mathbf{P}\left(R_\mathcal{G}\left(0,B_\mathcal{G}(0,rn\tilde{\psi}(n))^c\right)\geq \Lambda n\tilde{\psi}(n)\right)\geq \mathbf{P}\left(r\geq \Lambda+1\right).\]
Taking $r\geq \Lambda+1$ results in the right-hand side being equal to one, and so we obtain the desired result.
\end{proof}

We are nearly ready to complete the proof of Theorem \ref{main1}. It remains to introduce the resistance form framework in which \cite[Theorem 7.2]{Croyres} is stated. (We highlight that the idea of a resistance form was originally pioneered by Kigami in the context of analysis on self-similar fractals, see \cite{KAOF}.) To this end, we let $\mathbb{F}^*$ be quintuplets of the form $(F,R,\mu,\rho,f)$, where:
\begin{itemize}
  \item $F$ is a non-empty set;
  \item $R$ is a resistance metric on $F$ (i.e.\ for each finite $V\subseteq F$, there exist conductances between the elements of $V$ for which $R|_{V\times V}$ is the associated effective resistance metric, see \cite[Definition 2.3.2]{KAOF}), and closed bounded sets in $(F,R)$ are compact;
  \item $\mu$ is a locally finite Borel regular measure of full support on $(F,R)$;
  \item $\rho$ is a distinguished point in $F$;
  \item $f$ is a continuous map from $(F,R)$ into some image space $(M, d_M)$, which is a fixed complete, separable metric space.
\end{itemize}
We highlight that the resistance metric is associated with a certain quadratic form known as a resistance form $(\mathcal{E},\mathcal{F})$, as characterised by
\[R(x,y)^{-1}=\inf\left\{\mathcal{E}(u,u):\:u\in\mathcal{F},\:u(x)=0,\:u(y)=1\right\},\qquad \forall x,y\in F,\:x\neq y,\]
(see \cite[Section 2.3]{KAOF},) and we will further assume that for elements of $\mathbb{F}$ this form is regular in the sense of \cite[Definition 6.2]{Kq}. In particular, this ensures the existence of a related regular Dirichlet form $(\mathcal{E},\mathcal{D})$ on $L^2(F,\mu)$ (see \cite[Theorem 9.4]{Kq}), which we suppose is recurrent. The following result shows that the spaces introduced above fall into this class.

\begin{lem}\label{flem} $\mathbf{P}$-a.s., $\mathcal{X}_n$, $\tilde{\mathcal{X}}_n$, $\mathcal{X}$ and $\tilde{\mathcal{X}}$ are all elements of $\mathbb{F}^*$.
\end{lem}
\begin{proof} By construction, $R_\mathcal{G}$ is the resistance metric associated with the resistance form
\[\mathcal{E}_\mathcal{G}(u,v):=\frac{1}{2}\sum_{\substack{x,y\in V(\mathcal{G}):\\\{x,y\}\in E(\mathcal{G})}}(u(x)-u(y))(v(x)-v(y)),\qquad \forall u,v\in\mathcal{F}_\mathcal{G},\]
where $\mathcal{F}_\mathcal{G}$ is the collection of functions on $V(\mathcal{G})$ such that $\mathcal{E}_\mathcal{G}(u,u)<\infty$, see \cite[Example 3.5]{Kq}. Moreover, as was commented above Proposition \ref{p1}, $R_\mathcal{G}$ balls are finite, and so compact. Not only does this confirm $(F,R)$ and $\mu$ satisfy the properties in the first three bullet points above, but it also implies that $\mathcal{F}_\mathcal{G}$ contains all compactly supported functions. The latter observation gives that $(\mathcal{E}_\mathcal{G},\mathcal{F}_\mathcal{G})$ is regular. As for the recurrence of the associated Dirichlet form $(\mathcal{E}_\mathcal{G},\mathcal{D}_\mathcal{G})$ on $L^2(V(\mathcal{G},\mu_{\mathcal{G}})$, this is equivalent to
\begin{equation}\label{rcond}
\lim_{r\rightarrow\infty}R_\mathcal{G}\left(0,B_\mathcal{G}(0,r)^c\right)=\infty,
\end{equation}
see \cite[Lemma 2.3]{Croyres}, for instance. This divergence is given by Proposition \ref{p2}. Hence we obtain that $\mathcal{X}_1$ and $\tilde{\mathcal{X}}_1$ are in $\mathbb{F}^*$, $\mathbf{P}$-a.s. The rescaled spaces $\mathcal{X}_n$ and $\tilde{\mathcal{X}}_n$ can be dealt with similarly.

By \cite[Proposition 16.4]{Kq}, $(\mathbb{R},d_E)$ is a resistance metric space with associated regular resistance form
\[\mathcal{E}^{(2)}(u,v):=\int_{-\infty}^\infty u'(x)v'(x)dx,\qquad \forall u,v\in\mathcal{F}^{(2)},\]
where $\mathcal{F}^{(2)}$ is the collection of $u$ in $C(\mathbb{R})$ such that there exists an $\mathcal{E}^{(2)}$-Cauchy sequence of functions in $\mathcal{F}_0^{(2)}:=\{v:\:v\in C^1(\mathbb{R}),\:\int_{-\infty}^\infty v'(x)^2dx<\infty\}$ that converge uniformly on compacts to $u$. By taking the trace onto the one-sided space $[0,\infty)$ as per \cite[Theorem 8.4]{Kq}, we find that $([0,\infty),d_E)$ is also a resistance metric space associated with a regular resistance form $(\tilde{\mathcal{E}}^{(2)},\tilde{\mathcal{F}}^{(2)})$. To check the recurrence of the associated Dirichlet form $(\tilde{\mathcal{E}}^{(2)},\tilde{\mathcal{D}}^{(2)})$ on $L^2([0,\infty),\mathcal{L})$, we again appeal to \eqref{rcond}. In this case, we have that the resistance from $0$ to the boundary of the Euclidean ball of radius $r$ centred at 0 is simply given by $r$, and hence \eqref{rcond} holds. Since the remaining conditions are straightforward to check, we have thus established that $\mathcal{X}$ and $\tilde{\mathcal{X}}$ lie in $\mathbb{F}^*$, $\mathbf{P}$-a.s.
\end{proof}

Now, since the Dirichlet forms elements associated with elements of $\mathbb{F}^*$ are regular, they are naturally associated with Hunt processes. In particular, applying the spatial embeddings, the random processes corresponding to $\mathcal{X}_n$ and $\tilde{\mathcal{X}}_n$ are given by
\[\left(\frac{1}{n\phi(n)}d_\mathcal{G}\left(0,X^{cont}_{\lfloor tn^2\psi(n)\rfloor}\right)\right)_{t\geq 0},\qquad\left(\frac{1}{n^{1/2}}X^{cont}_{\lfloor tn^2\psi(n)\rfloor}\right)_{t\geq 0},\]
respectively, where $(X^{cont}_t)_{t\geq0}$ is the continuous-time random walk on $V(\mathcal{G})$ with jump chain equal to $(X_n)_{n\geq 0}$ and mean one exponential holding times (see \cite[Remark 5.7]{Barlow}, for example). Moreover, the random processes corresponding to $\mathcal{X}$ and $\tilde{\mathcal{X}}$ are given by
\[\left(|B_t|\right)_{t\geq 0},\qquad\left(W_{|B_t|}\right)_{t\geq 0},\]
respectively, where we use the fact that $|B|$ has the same law as reflected Brownian motion on $[0,\infty)$ (see \cite[Example 8.1]{AEW}, as well as \cite[Remarks 1.6 and 3.1]{AEW} for the connection between the framework of that paper and that of Kigami). With these preparations in place we are now ready to proceed with the proof of the main result.

\begin{proof}[Proof of Theorem \ref{main1}]
Given Propositions \ref{p1} and \ref{p2}, Lemma \ref{flem}, and the descriptions of the processes associated to the resistance spaces that precede this proof, applying \cite[Theorem 7.2]{Croyres} immediately yields that the conclusions of Theorem \ref{main1} hold when $X$ is replaced by its continuous-time counterpart $X^{cont}$. Since $X^{cont}$ has jump rate one from every vertex and the limiting processes are continuous, it is then straightforward to obtain the same results for $X$.
\end{proof}

We conclude the article by checking the further random walk properties that are stated as Corollary \ref{srwcor}.

\begin{proof}[Proof of Corollary \ref{srwcor}]
Defining $R(n,1)$ as at \eqref{rnrdef}, we have that $S_{[0,T_{R(n,1)}]}\subseteq B_\mathcal{G}(0,n\tilde{\psi}(n))\subseteq S_{[0,T_{R(n,1)+1}]}$. By Theorem \ref{main2}(d) and \eqref{e0}, it holds that $n^{-1}T_{R(n,1)}$ and $n^{-1}T_{R(n,1)+1}$ both converge to one in $\mathbf{P}$-probability. Together with Theorem \ref{main2}(a), we thus obtain that
\begin{equation}\label{muglim}
\frac{\mu_\mathcal{G}\left(B_\mathcal{G}\left(0,n\tilde{\psi}(n)\right)\right)}{n}\to \lambda
\end{equation}
in $\mathbf{P}$-probability. A reparameterisation applying \cite[Remark 2.1.3]{S} (which implies that $\tilde{\psi}(n)\sim\tilde{\psi}(n\tilde{\psi}(n))$) then gives that
\[\frac{\mu_\mathcal{G}\left(B_\mathcal{G}\left(0,n\right)\right)}{n\tilde{\psi}(n)^{-1}}\to \lambda\]
in $\mathbf{P}$-probability. Combined with Proposition \ref{p2}, this establishes that \cite[Assumption 1.2(1)]{KM} holds in our setting with $d=R_\mathcal{G}$, $v(x)=x\tilde{\psi}(x)^{-1}$ and $r(x)=x$. Therefore we can apply \cite[Proposition 1.3]{KM} to deduce that
\[\lim_{\Lambda\rightarrow\infty}\inf_{r\geq 1}\mathbf{P}\left(\Lambda^{-1}\leq\frac{\mathbf{E}_0^\mathcal{G}\left(\tilde{\tau}^\mathcal{G}_r\right)}{r^2\tilde{\psi}(r)^{-1}}\leq\Lambda\right)=1,\]
where $\tilde{\tau}_r^\mathcal{G}:=\inf\{n\geq 0:\:R_\mathcal{G}(0,X_n)\geq r\}$. Now, by Theorem \ref{main2}(b),(c), a $d_\mathcal{G}$-ball of radius $n\phi(n)$ is comparable with an $R_\mathcal{G}$-ball of radius $n\tilde{\psi}(n)$. Or, after reparameterisation, a $d_\mathcal{G}$-ball of radius $r$ is comparable with an $R_\mathcal{G}$-ball of radius $r\tilde{\psi}(r)\phi(r)^{-1}$. Hence we obtain part (a) of the result.

Towards proving part (b), we start by considering the continuity of the heat kernel. In particular, applying \cite[Proposition 12]{LLT}, we have that
\begin{equation}\label{b1}
\sup_{t\in I}\left|\lambda np_{tn^2\psi(n)}^\mathcal{G}\left(0,S_{ nx}\right)-\lambda np_{ tn^2\psi(n)}^\mathcal{G}\left(0,S_{ ny}\right)\right|^2\leq \frac{C R_\mathcal{G}\left(S_{ nx},S_{ny}\right)h_\mathcal{G}^{-1}\left(Cn^2\psi(n)\right)}{n^2\psi(n)^2},\:\forall x,y\geq 0,
\end{equation}
where $h_\mathcal{G}^{-1}(t):=\sup\{r:\:r\mu_\mathcal{G}(B_\mathcal{G}(0,r))\leq t\}$ and $C$ is a deterministic constant that only depends on $I$. Now, from \eqref{muglim}, we have that
\begin{equation}\label{b2}
\frac{h_\mathcal{G}^{-1}\left(Cn^2\psi(n)\right)}{n\psi(n)}\rightarrow c
\end{equation}
in $\mathbf{P}$-probability, for some deterministic constant $c$ (depending only on $C$). Putting together \eqref{b1}, \eqref{b2}, and Theorem \ref{main2}(b), it follows that: for any $x_0\geq 0$ and $\varepsilon\in(0,1)$,
\begin{equation}\label{b3}
\sup_{t\in I}\sup_{\substack{x,y\in[0,x_0+1]:\\|x-y|\leq \varepsilon}}\left|\lambda np_{tn^2\psi(n)}^\mathcal{G}\left(0,S_{ nx}\right)-\lambda np_{ tn^2\psi(n)}^\mathcal{G}\left(0,S_{ ny}\right)\right|^2\leq C\varepsilon,
\end{equation}
with $\mathbf{P}$-probability converging to one as $n\to\infty$, where $C$ is a deterministic constant that depends on $I$, but not $\varepsilon$. The remainder of the proof is similar to that of \cite[Theorem 1]{LLT}. In particular, we decompose the expression we are trying to bound as follows: for any $\delta>0$,
\[\left|\lambda np_{ tn^2\psi(n)}^\mathcal{G}\left(0,S_{ nx}\right)-p^{|B|}_t(x)\right|\leq \left|\lambda np_{ tn^2\psi(n)}^\mathcal{G}\left(0,S_{ nx}\right)-T_1\right|+\left|T_1-T_2\right|+\left|T_2-p^{|B|}_t(x)\right|,\]
where
\[p^{|B|}_t(x):=\sqrt{\frac{2}{\pi t}}e^{-\frac{x^2}{2t}}\]
is the transition density of the process $|B|$, and
\[T_1:=\frac{\lambda n\left(\mathbf{P}_0^\mathcal{G}\left(d_E\left(x,\frac{d_\mathcal{{G}}\left(0,X_{tn^2\psi(n)}\right)}{n\phi(n)}\right)\leq \delta\right)+\mathbf{P}_0^\mathcal{G}\left(d_E\left(x,\frac{d_\mathcal{{G}}\left(0,X_{tn^2\psi(n)+1}\right)}{n\phi(n)}\right)\leq \delta\right)\right)}{2\mu_\mathcal{G}\left(\left\{y\in V(\mathcal{G}):\:d_E\left(x,\frac{d_\mathcal{G}(0,y)}{n\phi(n)}\right)\leq \delta\right\}\right)},\]
\[T_2:=\frac{\mathbf{P}\left(d_E\left(x,|B_t|\right)\leq \delta\right)}{\int_0^\infty \mathbf{1}_{\{d_E(x,y)\leq\delta\}}dy}.\]
By Theorem \ref{main2}(c), we know that
\[\left\{y\in V(\mathcal{G}):\:d_E\left(x,\frac{d_\mathcal{G}(0,y)}{n\phi(n)}\right)\leq \delta\right\}
\subseteq\left\{S_m:\:\max\{n(x-2\delta),0\}\leq m\leq n(x+2\delta)\right\}\]
for all $x\in[0,x_0]$ with $\mathbf{P}$-probability converging to one. Together with \eqref{b3}, if $\delta\leq \varepsilon/2$, this implies
\[\sup_{t\in I}\sup_{x\in[0,x_0]}\left|\lambda np_{ tn^2\psi(n)}^\mathcal{G}\left(0,S_{ nx}\right)-T_1\right|\leq C\varepsilon.\]
with $\mathbf{P}$-probability converging to one as $n\to\infty$. Similarly, the continuity of $p^{|B|}_t(x)$ allows us to deduce that, if $\delta$ is chosen sufficiently small, then $|T_2-p^{|B|}_t(x)|\leq C\varepsilon$. Finally, we note that Theorems \ref{main1}(a) and \ref{main2}(a),(c) imply that $T_1\rightarrow T_2$ in $\mathbf{P}$-probability. The result follows.
\end{proof}

\section*{Acknowledgements}

DC was supported by JSPS Grant-in-Aid for Scientific Research (A), 17H01093, JSPS Grant-in-Aid for Scientific Research (C), 19K03540, and the Research Institute for Mathematical Sciences, an International Joint Usage/Research Center located in Kyoto University. DS was supported by a JSPS Grant-in-Aid for Early-Career Scientists, 18K13425, JSPS  Grant-in-Aid for Scientific Research (B), 17H02849, and JSPS  Grant-in-Aid for Scientific Research (B), 18H01123.

\bibliography{4dRWRRW}
\bibliographystyle{amsplain}

\end{document}